\documentclass[a4paper,11pt]{article}
\usepackage{amsmath,amsthm,amssymb,fullpage,enumerate,tikz}
\usepackage{xcolor}
\definecolor{internallinkcolor}{rgb}{0,.5,0}
\definecolor{externallinkcolor}{rgb}{0,0,.5}
\usepackage[colorlinks=true,
urlcolor=externallinkcolor,
linkcolor=internallinkcolor,
filecolor=externallinkcolor,
citecolor=internallinkcolor,
]{hyperref}
\usepackage[capitalise]{cleveref}
\usepackage[nobysame,msc-links,non-sorted-cites]{amsrefs}

\usepackage{thmtools}
\declaretheorem[parent=section]{theorem}
\declaretheorem[sibling=theorem]{lemma}
\declaretheorem[sibling=theorem]{corollary}
\declaretheorem[sibling=theorem]{observation}
\declaretheorem[sibling=theorem]{conjecture}

\newcommand{\floor}[1]{\lfloor#1\rfloor}
\newcommand{\ceil}[1]{\left\lceil #1\right\rceil}
\newcommand{\defn}[1]{\textbf{#1}}
\newcommand{\sm}{\smallsetminus}
\DeclareMathOperator{\rig}{rig}
\newcommand{\maxwell}{\mathsf{M}}

\title{Rigidity of complements of bounded-degree graphs}
\author{John Haslegrave, Peleg Michaeli and Anthony Nixon}

\begin{document}

\maketitle

\begin{abstract}
Maxwell observed that the graph of any rigid generic framework in $\mathbb{R}^d$ on $n$ vertices has at least $dn-\binom{d+1}{2}$ edges.
In this article we prove that graphs whose complement has maximum degree at most two and no component isomorphic to a triangle or a square
are rigid in the maximum dimension allowed by this observation.
In particular, this
determines the precise maximum dimension
in which the graph obtained from a complete graph $K_{2m}$ by deleting a perfect matching is rigid,
resolving a recent conjecture of Lew. We also deduce bounds on the rigidity of complements of bounded-degree graphs more generally, which significantly improve existing degree-based bounds.
\end{abstract}

\section{Introduction}

A \defn{bar-joint framework} $(G,p)$ is an ordered pair consisting of a finite, simple graph $G=(V,E)$ and a realisation $p:V\rightarrow \mathbb{R}^d$.
The framework is \defn{rigid}, informally speaking, if the only edge-length-preserving continuous deformations of the vertices arise from isometries of $\mathbb{R}^d$.
It is a fundamental problem in discrete geometry, motivated by numerous practical applications, to determine whether a given framework is rigid.
However, in general this is computationally infeasible \cite{Abb08}.
Thus attention focuses on the generic case.
A framework is \defn{generic} if the coordinates of the points $p(v):v\in V$ form an algebraically independent set over $\mathbb{Q}$.
It is well known that a generic framework $(G,p)$ in $\mathbb{R}^d$ is rigid if and only if every other generic framework $(G,q)$ in $\mathbb{R}^d$ is rigid \cite{AR78}.
In this sense rigidity depends only on the graph and the dimension.

Determining which graphs are rigid in a given dimension is a well-studied problem.
When $d\leq 2$ complete descriptions are known \cites{Lam70,PG27} whereas when $d\geq 3$ it is the central open problem in rigidity theory to give a combinatorial characterisation of rigid graphs.
Partial results are known from various directions; see \cites{Clinch2022,HJNS25+,JLV25+,Katoh2011,KLM26} inter alia.
We add to this growing literature on special cases by proving the following conjecture of Lew \cite{CN23}*{Conjecture~3.4} and some generalisations.

Put $a_0=0$. For each integer $j\ge 1$, let $a_j$ be the unique integer satisfying
$\binom{a_j}{2}<j\le\binom{a_j+1}{2}$.
Equivalently, for every $j\ge 0$,
\begin{equation}\label{eq:triangular-index}
  a_j=\floor{\sqrt{2j}+1/2}.
\end{equation}

\begin{conjecture}[Lew]\label{con:1}
  For all $r\ge 3$, the graph obtained from $K_{2r}$ by deleting a perfect matching is generically
  $(2r-1-a_r)$-rigid.
\end{conjecture}

The case $r=3$ is the octahedron which, as a triangulation of the sphere, is generically $3$-rigid (see \cite{Glu75}).
For this reason, the corresponding graph for larger $r$ is sometimes called the \emph{hyperoctahedral graph}.

We regard every graph as generically $0$-rigid. 
For an $n$-vertex graph $G$,
define its \defn{rigidity dimension} $\rig(G)$ to be the largest nonnegative integer $d\leq n-1$
for which $G$ is generically $d$-rigid.
Define its \defn{Maxwell dimension} $\maxwell(G)$ to be the largest nonnegative integer $d\leq n-1$ satisfying
\mbox{$dn-\binom{d+1}{2}\leq |E(G)|$}.
For graphs of the same order, the Maxwell dimension is nondecreasing in the number of edges.
Maxwell's edge count gives $\rig(G)\leq\maxwell(G)$.
We say that $G$ is \defn{Maxwell-sharp} if $\rig(G)=\maxwell(G)$.
Thus \cref{con:1} says that complements of $1$-regular graphs are Maxwell-sharp.

The following conjecture of Krivelevich, Lew and Michaeli~\cite{KLM26} extends \cref{con:1} by predicting when the two basic necessary conditions, vertex connectivity and Maxwell's count, are sufficient.
\begin{conjecture}\label{con:KLM}
  Let $n>d\ge 1$, and let $G$ be an $n$-vertex graph.
  If
  \[
    \delta(G)\ge
    \max\left\{\frac{n+d}{2}-1,\,
    2d-\frac{d(d+1)}{n}\right\},
  \]
  then $G$ is generically $d$-rigid.
\end{conjecture}

The two expressions in the maximum correspond, respectively, to vertex connectivity and Maxwell's count.

Krivelevich, Lew and Michaeli~\cite{KLM26} proved \cref{con:KLM} when \mbox{$d\leq(\sqrt{8n-15}-1)/4$},
and proved that the stronger assumption $\delta(G)\geq n/2+d-1$ suffices when $d\leq cn/\log^2 n$ for some absolute $c>0$.
Jord\'an, Liu and Vill\'anyi~\cite{JLV25+} subsequently proved that the latter assumption suffices for all $d$,
and proved \cref{con:KLM} when $d\leq n/29$.
Specialising \cref{con:KLM} to the dense regular regime gives the following conjecture.

\begin{conjecture}
\label{con:KLM-regular}
  Every $k$-regular $n$-vertex graph with \mbox{$k\geq(3n-6)/4$} is Maxwell-sharp.
\end{conjecture}

For a $k$-regular graph, Maxwell's edge-count condition in dimension $d$ is equivalent to the inequality $k\geq 2d-d(d+1)/n$ in \cref{con:KLM}.
When $k\geq(3n-6)/4$, the other inequality in \cref{con:KLM}, namely $k\geq(n+d)/2-1$, is automatically satisfied for $d=\maxwell(G)$.
Thus \cref{con:KLM-regular} is precisely the regular dense case of \cref{con:KLM}.
The case $k=n-1$ is trivial.
The first nontrivial case, $k=n-2$, is precisely \cref{con:1}.
Our main result gives a complete resolution of \cref{con:1}.
More generally, it proves the stronger Maxwell-sharp conclusion in the \mbox{$\delta(G)\geq n-3$} regime of \cref{con:KLM},
apart from graphs whose complement has a component isomorphic to $C_3$ or $C_4$.
Beyond these results, \cref{con:KLM} remains wide open, even in the regular dense setting of \cref{con:KLM-regular}.

\begin{theorem}\label{thm:2-reg}
  Let $n\geq 4$, and let $G$ be an $n$-vertex graph with $\delta(G)\geq n-3$.
  Suppose no connected component of the complement of $G$ is isomorphic to $C_3$ or $C_4$.
  Then $G$ is Maxwell-sharp.
\end{theorem}

\Cref{thm:2-reg} includes \cref{con:1}, since, when $n$ is even, the complement of an $(n-2)$-regular graph is a perfect matching.
In the $(n-3)$-regular case, it says that the complement of every $2$-factor with no component isomorphic to $C_3$ or $C_4$ is Maxwell-sharp.

We believe that the exclusion of $C_3$ and $C_4$ is a limitation of our argument rather than a genuine obstruction.
This limitation is inherited from the non-twin hypothesis in a theorem of Hewetson, Jackson, Nixon and Smith~\cite{HJNS25+},
which we restate below as \cref{thm:hjns2};
that hypothesis fails in the final step of our argument for a $C_3$- or $C_4$-component.

A direct extension of \cref{thm:2-reg} to complements of all $3$-regular graphs is false.
Indeed, the complement of $K_{3,3}$ satisfies Maxwell's edge-count condition in dimension $1$ but is disconnected, and hence is not generically $1$-rigid.
This example does not contradict \cref{con:KLM} or \cref{con:KLM-regular}.
It has $6$ vertices and minimum degree $2$, whereas \cref{con:KLM} requires minimum degree at least $5/2$ in dimension $1$, and \cref{con:KLM-regular} applies only when \mbox{$\delta\geq 3$}.
We continue to believe both conjectures.

For a graph $F$, we write $\overline F$ for its complement.
We also call a set $S\subseteq V(F)$ \defn{admissible} if $F-S$ has maximum degree at most two
and has no connected component isomorphic to $C_3$ or $C_4$.
Define
\[
  \sigma(F)=\min\{|S|:S\subseteq V(F)\text{ is admissible}\}.
\]
Thus $|V(F)|-\sigma(F)$ is the maximum order of an induced subgraph of $F$
whose components are paths, possibly trivial, or cycles of length at least five.

\begin{corollary}\label{cor:exceptional-components}
  Let $n\geq 4$, and let $G$ be an $n$-vertex graph with $\delta(G)\geq n-3$.
  Then
  \[
    \rig(G)\geq\maxwell(G)-\bigl\lceil\sigma(\overline G)/2\bigr\rceil.
  \]
\end{corollary}

Using a theorem of Joret and Petit~\cite{JP25} on induced linear forests,
we also obtain the following bound for complements of bounded degree.

\begin{corollary}\label{cor:bounded-complement-degree}
  Let $2\leq\Delta=o(\sqrt n)$,
  and let $G$ be an $n$-vertex graph with $\delta(G)\geq n-1-\Delta$.
  Then
  \[
    \rig(G)\geq\frac n2+\frac{n}{\Delta+1}-O(\sqrt n+\Delta^2).
  \]
\end{corollary}

The following asymptotic consequence of \cref{con:KLM} may be a more approachable first step in its high-dimensional regime.
\begin{conjecture}\label{con:bounded-complement-degree}
  Let $D=D(n)=o(n)$.
  Then, uniformly over all $n$-vertex graphs $F$ with $\Delta(F)\leq D$,
  \[
    \rig(\overline{F})=(1-o(1))n.
  \]
\end{conjecture}

For comparison, Jord\'an, Liu and Vill\'anyi~\cite{JLV25+} proved that the condition $\delta(G)\geq n/2+q-1$ suffices for generic $q$-rigidity.
Under the minimum-degree assumption in \cref{cor:bounded-complement-degree},
their theorem guarantees generic $q$-rigidity only for $q\leq\floor{n/2}-\Delta$ using this hypothesis alone.
Indeed, since $\delta(G)\leq n-1$,
their condition can never certify rigidity in a dimension greater than $\floor{n/2}$,
regardless of how dense the graph is.
By contrast, if $\Delta=o(n^{1/3})$,
then \cref{cor:bounded-complement-degree} gives rigidity in dimension
\[
  \frac n2+(1-o(1))\frac{n}{\Delta+1}.
\]
For the graphs in \cref{cor:exceptional-components},
taking $\Delta=2$ in \cref{cor:bounded-complement-degree} gives generic rigidity in dimension $5n/6-O(\sqrt n)$ uniformly.
On the other hand,
\cref{cor:exceptional-components} retains the dependence on $\sigma(\overline G)$.
Since $\maxwell(G)=n-O(\sqrt n)$,
it gives generic rigidity in dimension
$n-\lceil\sigma(\overline G)/2\rceil-O(\sqrt n)$.
More precisely, \cref{cor:exceptional-components} gives
\mbox{$\maxwell(G)-\rig(G)\leq\lceil\sigma(\overline G)/2\rceil$}.

\medskip

The paper is organised as follows.
In \cref{sec:rigid-graphs} we collect a few rigidity results that we will need.
In \cref{sec:two-factors} we prove \cref{thm:2-reg}.
In \cref{sec:bounded-degree} we apply our results to complements of graphs of bounded degree, deducing in particular \cref{cor:exceptional-components} and \cref{cor:bounded-complement-degree}.

\section{Rigidity preliminaries}\label{sec:rigid-graphs}

Let $\mathcal{M}=(E,r)$ be a matroid with finite ground set $E$ and rank function $r$.
A \defn{circuit} of $\mathcal{M}$ is a set $C\subseteq E$ such that
$r(C)=|C|-1=r(C-e)$ for all $e\in C$.

Let $G=(V,E)$ be a graph.
A \defn{($d$-dimensional) framework} of $G$ is a pair $(G,p)$,
where $p:V\to \mathbb{R}^d$.
The framework is \defn{rigid} if every continuous motion of the vertices
preserving the lengths of the edges preserves all pairwise distances.
The map $p$ is \defn{generic} if the $d|V|$ coordinates of $p(V)$
are algebraically independent over $\mathbb Q$.
We say that $G$ is \defn{generically $d$-rigid}
if $(G,p)$ is rigid for some, equivalently every, generic map
$p:V\to \mathbb{R}^d$.
When no confusion is possible,
we simply say that $G$ is \defn{$d$-rigid}.
It is standard that if $G$ is $d$-rigid,
then it is $d'$-rigid for every $0\leq d'\leq d$.

Define the \defn{rigidity matrix} $R(G,p)$ of a $d$-dimensional framework $(G,p)$
to be the $|E|\times d|V|$ matrix whose rows are indexed by the edges,
and $d$-tuples of columns are indexed by the vertices.
The row for an edge $e=uv$ is given by:
\[\begin{pmatrix} 0 & \cdots & 0 & p(u)-p(v) & 0 & \cdots & 0 & p(v)-p(u) & 0 & \cdots & 0  \end{pmatrix}\]
where $p(u)-p(v)$ occurs in the $d$-tuple of columns indexed by $u$,
$p(v)-p(u)$ occurs in the $d$-tuple of columns indexed by $v$.
Maxwell~\cite{Max90} observed that
$\operatorname{rank}(R(G,p))\leq d|V|-\binom{d+1}{2}$
whenever $p$ affinely spans $\mathbb{R}^d$,
and a fundamental result of Asimow and Roth~\cite{AR78} tells us that,
when $p$ is generic and $G$ has at least $d+1$ vertices,
$(G,p)$ is rigid if and only if
$\operatorname{rank}(R(G,p))=d|V|-\binom{d+1}{2}$.

The \defn{$d$-dimensional rigidity matroid} of a graph $G=(V,E)$
is the matroid $\mathcal{R}_d(G)$ on $E$ in which a set of edges $F\subseteq E$
is independent if the corresponding rows of $R(G,p)$ are independent,
for some (or equivalently every) generic $p:V\to \mathbb{R}^d$.
We denote the rank of $\mathcal{R}_d(G)$ by $r_d(G)$.

We will simplify terminology by describing $G$ using properties of its edge set
in $\mathcal{R}_d(G)$.
For example we say that $G$ is:
\defn{$\mathcal{R}_d$-independent} if $r_d(G)=|E|$;
\defn{$\mathcal{R}_d$-rigid} if $G$ is a complete graph on at most $d+1$ vertices
or $r_d(G)=d|V|-\binom{d+1}{2}$;
\defn{minimally $\mathcal{R}_d$-rigid} if $G$ is $\mathcal{R}_d$-rigid
and $\mathcal{R}_d$-independent;
and an \defn{$\mathcal{R}_d$-circuit} if $G$ is not $\mathcal{R}_d$-independent
but $G-e$ is $\mathcal{R}_d$-independent for all $e\in E$.
If $G$ is $\mathcal{R}_d$-rigid with at least $d+1$ vertices,
we say that it has $|E(G)|-d|V(G)|+\binom{d+1}{2}$ \defn{excess edges}.
When no confusion is possible, we write
\defn{$d$-independent}, \defn{minimally $d$-rigid}, and \defn{$d$-circuit}
for the corresponding notions in $\mathcal{R}_d(G)$.
We say that a circuit $C$ \defn{meets} a vertex $v$ if there is some edge
$e$ with $v\in e\in C$.

\begin{lemma}[$0$-extension; see, for example, \cite{TW85}]
\label{lem:zero-extension}
  Let $G=(V,E)$ be a $d$-rigid graph,
  and let $G'$ be obtained from $G$ by adding a new vertex $v$
  adjacent to at least $d$ vertices of $V$.
  Then $G'$ is $d$-rigid.
\end{lemma}

We will also use the following form of Whiteley's vertex-splitting theorem.
\begin{lemma}[Vertex splitting; \cite{Whi90}*{Corollary~11}]
\label{lem:vertex-splitting}
  Let $G=(V,E)$ be a $d$-rigid graph,
  and let $v\in V$.
  Let $G'$ be obtained from $G$ by removing $v$ and adding two adjacent vertices $x$ and $y$ such that
  \[
    N_{G'}(x)\cup N_{G'}(y)=N_G(v)\cup\{x,y\}
  \]
  and $|N_{G'}(x)\cap N_{G'}(y)|\geq d-1$.
  Then $G'$ is $d$-rigid.
\end{lemma}

\subsection{Coning}

The \defn{cone} $\hat G$ of a graph $G$ is obtained by adding a new vertex
and joining it to every vertex of $G$.
We will use classical and recent results about coning.

\begin{theorem}[\cite{Whi83}]\label{thm:coning}
  Let $\hat G$ be the cone of a graph $G$.
  Then $r_{d+1}(\hat G)=r_d(G)+|V(G)|$.
  In particular, $\hat G$ is $(d+1)$-rigid if and only if $G$ is $d$-rigid.
\end{theorem}

Combining coning, vertex splitting and 0-extension gives the following.
\begin{lemma}\label{lem:paired-extension}
  Let $G$ be a graph, let $S\subseteq V(G)$,
  and put $H=G-S$, $n'=|V(H)|$ and $F=\overline G$.
  Suppose that $H$ is $d$-rigid,
  where $0\leq d\leq n'-2\Delta(F)$.
  Let $M$ be a matching on $S$ such that the ends of every edge of $M$
  have distance at least $3$ in $F$.
  Then $G$ is $(d+|M|)$-rigid.
\end{lemma}

\begin{proof}
  Put $m=|M|$ and $k=|S\sm V(M)|$.
  If $d=m=0$, the result follows from the convention that every graph is $0$-rigid.
  Write $M=\{x_iy_i:1\leq i\leq m\}$
  and $S\sm V(M)=\{z_1,\ldots,z_k\}$.
  Put $G_0=H$ and, for $1\leq i\leq m$, put
  \[
    G_i=G[V(H)\cup\{x_j,y_j:1\leq j\leq i\}].
  \]
  We show inductively that $G_i$ is $(d+i)$-rigid.
  Suppose that $1\leq i\leq m$ and that $G_{i-1}$ is $(d+i-1)$-rigid.
  Since $x_i$ and $y_i$ have distance at least $3$ in $F$,
  they are non-adjacent and have no common neighbour in $F$.
  Thus they are adjacent in $G$,
  and their neighbourhoods in $G$ cover $V(G_{i-1})$.
  If $d=0$ and $i=1$, then $G_1$ is connected and hence $1$-rigid.
  Otherwise, $x_i$ and $y_i$ have at least
  $n'+2i-2-2\Delta(F)\geq d+i-1$ common neighbours in $V(G_{i-1})$.
  Thus \cref{thm:coning,lem:vertex-splitting} show that $G_i$ is $(d+i)$-rigid.

  For $1\leq i\leq k$, put
  \[
    G_{m+i}=G[V(G_m)\cup\{z_1,\ldots,z_i\}].
  \]
  Each $z_i$ has at least $n'+2m-\Delta(F)\geq d+m$ neighbours in $G_{m+i-1}$.
  Hence \cref{lem:zero-extension} implies inductively that every $G_{m+i}$ is $(d+m)$-rigid.
  Since $G_{m+k}=G$, the result follows.
\end{proof}

\begin{theorem}[\cite{HJNS25+}]\label{thm:hjns1}
  Let $G$ be a graph on at least $d$ vertices,
  let $\hat G$ be obtained from $G$ by coning with new vertex $v$
  and let $\hat G'$ be obtained from $\hat G$
  by deleting $vs_1$ for some $s_1\in V(G)$.
  Then $\hat G'$ is minimally $(d+1)$-rigid if and only if $G$ is $d$-rigid,
  $|E(G)| = d|V(G)|-\binom{d+1}{2}+1$ and
  $G-s_1$ is $d$-independent.
\end{theorem}

Dropping the minimality requirement in \cref{thm:hjns1} gives the following characterisation.

\begin{corollary}\label{thm:almostcone}
  Let $G$ be a graph on at least $d$ vertices,
  let $\hat G$ be obtained from $G$ by coning with new vertex $v$
  and let $\hat G'$ be obtained from $\hat G$
  by deleting $vs_1$ for some $s_1\in V(G)$.
  Then $\hat G'$ is $(d+1)$-rigid if and only if $G$ is
  $d$-rigid and some $d$-circuit in $G$ meets $s_1$. 
\end{corollary}

\begin{proof}
Suppose $G$ is $d$-rigid and some $d$-circuit $C$ in $G$ meets $s_1$.
If there is a $d$-circuit $C'$ distinct from $C$ then we may delete any single edge from $C'\sm C$, and repeat this process until no other $d$-circuit exists.
Since we only deleted edges from $d$-circuits, the resulting graph $G^-$ is still $d$-rigid.
Since it contains a unique $d$-circuit,
$|E(G^-)| = d|V(G^-)|-\binom{d+1}{2}+1$ and since $C$ meets $s_1$ we have that $G^- -s_1$ is $d$-independent.
Hence \cref{thm:hjns1} implies that $\hat G'$ is $(d+1)$-rigid.

Conversely, suppose $\hat G'$ is $(d+1)$-rigid.
If $\hat G'$ contains a $(d+1)$-circuit $C$ then necessarily $C$ contains edges not incident to $v$.
Delete any such edge, and repeat until we obtain $\hat H'$ from $\hat G'$ that is minimally $(d+1)$-rigid.
Let $H$ be obtained from $\hat H'$ by deleting $v$.
Then \cref{thm:hjns1} implies that $H$ is $d$-rigid, $|E(H)| = d|V(H)|-\binom{d+1}{2}+1$ and
  $H-s_1$ is $d$-independent.
Since $G$ is obtained from $H$ by adding edges, clearly $G$ is $d$-rigid and contains a $d$-circuit that meets $s_1$.
\end{proof}

\begin{theorem}[\cite{HJNS25+}]\label{thm:hjns2}
  Let $G$ be a graph on at least $d$ vertices, let $\hat G$ be obtained from $G$
  by coning with new vertex $v$
  and let $\hat G''$ be obtained from $\hat G$
  by deleting $vs_1,vs_2$ for some $s_1,s_2\in V(G)$.
  Suppose further that $N_G(s_1)-s_2\neq N_G(s_2)-s_1$.
  Then $\hat G''$ is minimally $(d+1)$-rigid if and only if $G$ is $d$-rigid,
  $|E(G)| = d|V(G)|-\binom{d+1}{2}+2$,
  $G-\{s_1,s_2\}$ is $d$-independent,
  and $s_{i}$ belongs to a $d$-circuit in $G$ for $i=1,2$.
\end{theorem}

If minimality is not required, we obtain the following sufficient condition.
Note that no necessary and sufficient condition purely in terms of the number of edges and the intersections of circuits with $\{s_1,s_2\}$ is possible, as Figure \ref{fig:no-nec-suff-cond} shows: there exist graphs (in fact the same graph with different choices of $s_1,s_2$) where the edge count and the number of circuits meeting each subset of $\{s_1,s_2\}$ are the same, but coning and removing the edges to $s_1$ and $s_2$ yields different rigidity properties. This also gives an example to show that the additional condition in (ii) that both circuits do not meet both vertices cannot be removed.

\begin{figure}
    \centering
\begin{tikzpicture}[scale=1.5]
    \filldraw (0,0) circle (0.05);
    \filldraw (1,0) circle (0.05);
    \filldraw (0,1) circle (0.05);
    \filldraw (1,1) circle (0.05);
    \filldraw (2,0) circle (0.05);
    \filldraw (2,1) circle (0.05);
    \filldraw (2.866,.5) circle (0.05);
    \filldraw (2.866,1.5) circle (0.05);
    \filldraw (3.866,.5) circle (0.05);
    \filldraw (3.866,1.5) circle (0.05);
    \node [anchor=east] at (0,0) {$s_1$};
    \node [anchor=east] at (0,1) {$s_2$};
    \node [anchor=south] at (1.5,2.25) {$v$};
    \draw (2.866,1.5)--(2,1)--(2,0)--(0,0)--(0,1)--(2,1)--(2.866,.5)--(3.866,.5)--(3.866,1.5)--cycle;
    \draw (2,0)--(1,1)--(1,0);
    \draw (2.866,.5)--(2.866,1.5);
    \draw [dashed] (1.5,2.25)--(1,0);
    \draw [dashed] (1.5,2.25)--(1,1);
    \draw [dashed] (1.5,2.25)--(2,0);
    \draw [dashed] (1.5,2.25)--(2,1);
    \draw [dashed] (1.5,2.25)--(2.866,.5);
    \draw [dashed] (1.5,2.25)--(2.866,1.5);
    \draw [dashed] (1.5,2.25)--(3.866,.5);
    \draw [dashed] (1.5,2.25)--(3.866,1.5);
    \filldraw[fill=white] (1.5,2.25) circle (0.05);

    \filldraw (5,0) circle (0.05);
    \filldraw (6,0) circle (0.05);
    \filldraw (5,1) circle (0.05);
    \filldraw (6,1) circle (0.05);
    \filldraw (7,0) circle (0.05);
    \filldraw (7,1) circle (0.05);
    \filldraw (7.866,.5) circle (0.05);
    \filldraw (7.866,1.5) circle (0.05);
    \filldraw (8.866,.5) circle (0.05);
    \filldraw (8.866,1.5) circle (0.05);
    \node [anchor=south east] at (7.866,1.5) {$s'_1$};
    \node [anchor=north east] at (7.866,.5) {$s'_2$};
    \node [anchor=south] at (6.5,2.25) {$v'$};
    \draw (7.866,1.5)--(7,1)--(7,0)--(5,0)--(5,1)--(7,1)--(7.866,.5)--(8.866,.5)--(8.866,1.5)--cycle;
    \draw (7,0)--(6,1)--(6,0);
    \draw (7.866,.5)--(7.866,1.5);
    \draw [dashed] (6.5,2.25)--(6,0);
    \draw [dashed] (6.5,2.25)--(6,1);
    \draw [dashed] (6.5,2.25)--(7,0);
    \draw [dashed] (6.5,2.25)--(7,1);
    \draw [dashed] (6.5,2.25)--(5,1);
    \draw [dashed] (6.5,2.25)--(5,0);
    \draw [dashed] (6.5,2.25)--(8.866,.5);
    \draw [dashed] (6.5,2.25)--(8.866,1.5);
        \filldraw[fill=white] (6.5,2.25) circle (0.05);
\end{tikzpicture}
\caption{The graph shown (solid lines) is $1$-rigid, with ten vertices, fourteen edges and nine 1-circuits. Three 1-circuits meet both $s_1$ and $s_2$, and six meet neither. The graph obtained by coning with new vertex $v$ and removing $vs_1,vs_2$ (additional dashed lines, left) is not $2$-rigid. Similarly, three 1-circuits meet both $s_1'$ and $s'_2$, 
and six meet neither. However, the graph obtained by coning with new vertex $v'$ and removing $v's_1',v's_2'$ (additional dashed lines, right) is $2$-rigid.}
\label{fig:no-nec-suff-cond}
\end{figure}
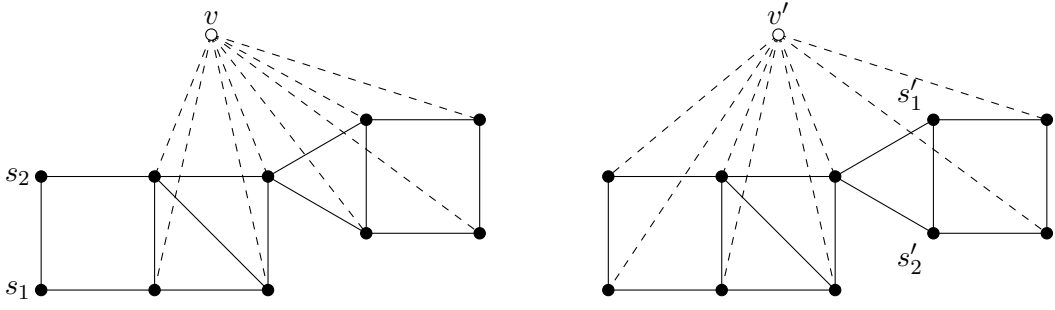

We first recall the strong circuit elimination property (see e.g.\ \cite{Oxley}*{Proposition 1.4.12}): if $C$ and $D$ are distinct circuits of a matroid, with elements $x\in C\cap D$ and $y\in C\sm D$, then there is a circuit $C'$ with $y\in C'\subseteq (C\cup D)-\{x\}$.

\begin{theorem}\label{thm:almostcone2}
 Let $G$ be a graph on at least $d$ vertices, let $\hat G$ be obtained from $G$
  by coning with new vertex $v$, let $\hat G''$ be obtained from $\hat G$
  by deleting $vs_1,vs_2$ for some $s_1,s_2\in V(G)$ and suppose that $N_G(s_1)-s_2\neq N_G(s_2)-s_1$.
  Suppose furthermore that $G$ is $d$-rigid,
  $|E(G)| \geq d|V(G)|-\binom{d+1}{2}+2$
  and one of the following holds: either
  \begin{enumerate}[(i)]\item all $d$-circuits of $G$ meet both $s_1$ and $s_2$; or
  \item $G$ has distinct $d$-circuits $C_1,C_2$
  such that $C_i$ meets $s_i$ for $i=1,2$, but $s_i$ does not meet $C_{3-i}$ for some $i\in \{1,2\}$.
  \end{enumerate}
  Then $\hat G''$ is $(d+1)$-rigid.
\end{theorem}

\begin{proof}
First suppose $d=0$ and note that a $0$-circuit is just an edge. Since $|E(G)|\geq 2$, (i) cannot apply, and (ii) requires distinct edges $e_1,e_2$ such that (without loss of generality) $e_1$ connects $s_1$ to $V(G)\sm\{s_1\}$ and $e_2$ connects $s_2$ to $V(G)\sm\{s_1,s_2\}$. In this case it is clear that $\hat G''$ is connected, i.e.\ $1$-rigid. Henceforth we assume $d\geq 1$.

Since $N_G(s_1)-s_2\neq N_G(s_2)-s_1$, there exists some edge $e$ meeting $s_1$ or $s_2$ such that $N_{G^-}(s_1)-s_2\neq N_{G^-}(s_2)-s_1$ for every spanning subgraph $G^-$ that contains $e$.

First, if (i) applies, remove edges other than $e$ from $d$-circuits to obtain a spanning subgraph $G^-$ containing $e$ with exactly $d|V(G)|-\binom{d+1}2+2$ edges.
This implies $G^-$ has at least two $d$-circuits, both of which meet $s_1$ and $s_2$, and $G^--\{s_1,s_2\}$ has no $d$-circuits.

Secondly, suppose (ii) applies, and choose $d$-circuits $C_1,C_2$ with this property such that $|C_1\cup C_2\cup\{e\}|$ is as small as possible.
If there is a $d$-circuit $C'\not\subset C_1\cup C_2\cup\{e\}$, then we may delete any single edge from $C'\sm (C_1\cup C_2\cup\{e\})$, and repeat this process until no such $d$-circuit exists.
Since we only deleted redundant edges, the resulting graph $G^-$ is $d$-rigid and contains $e$ and the $d$-circuits $C_1,C_2$.
Without loss of generality, we may assume $s_1$ does not meet $C_2$.

Next we claim that every $d$-circuit in $G^-$ meets either $s_1$ or $s_2$.
Indeed, suppose some $d$-circuit $C^*$ meets neither.
Every $d$-circuit of $G^-$ is contained in $C_1\cup C_2\cup\{e\}$, so $C^*\subseteq C_1\cup C_2\cup\{e\}$. Since $C_1$ meets $s_1$ we have $C^*\neq C_1$ so we may choose an edge $e^*\in C^*\sm C_1$. As $e^*$ lies in $C^*$ it is incident to neither $s_1$ nor $s_2$ and hence $e^*\in C_2$. Moreover $C_2\neq C^*$, since $C_2$ meets $s_2$ and $C^*$ does not, and for the same reason $C_2$ contains an edge $f$ incident to $s_2$ with $f\notin C^*$.
Applying strong circuit elimination to $C_2$ and $C^*$,
we obtain a $d$-circuit $C_2'$ with $f\in C_2'\subseteq (C_2\cup C^*)-\{e^*\}$.
Since $f\in C_2'$, the circuit $C_2'$ meets $s_2$. Since neither $C_2$ nor $C^*$ meets $s_1$, neither does $C_2'$ and hence $C_2'\neq C_1$.
However, now $C_1\cup C_2'\cup\{e\}\subseteq (C_1\cup C_2\cup\{e\})-\{e^*\}$. This contradicts the choice of $C_1,C_2$.
Thus $G^- - \{s_1,s_2\}$ is $d$-independent.

Finally, we claim that $G^-$ has exactly $d|V(G)|-\binom{d+1}2+2$ edges.
Indeed, choose edges $e_1\in C_1\sm C_2$ and $e_2\in C_2\sm C_1$, avoiding choosing $e$ if possible.
Since $C_1\sm C_2$ includes all edges of $C_1$ meeting $s_1$, and there are at least two of these, $e_1\neq e$.
Since $C_1,C_2$ are $d$-circuits, $C_2\sm C_1$ is non-empty, and we choose $e_2=e$ only if $C_2\sm C_1=\{e\}$.
Now every $d$-circuit of $G^-$ meeting $s_1$ must include all edges of $C_1\sm (C_2 \cup\{e\})$, since if some $d$-circuit $C_1'$ did not then $C_1'\cup C_2\cup\{e\}$ would be a proper subset of $C_1\cup C_2\cup\{e\}$, contradicting choice of $C_1$; in particular every such $d$-circuit includes $e_1$.
Similarly, if $e_2\neq e$ then every $d$-circuit of $G^-$ meeting $s_2$ but not $s_1$ includes $e_2$.
If $e_2=e$ then the same conclusion follows for a different reason: if some such $d$-circuit did not include $e$ then it would be a proper subset of $C_1$, which is impossible.
Since all remaining $d$-circuits fall into one of these types, $G^--\{e_1,e_2\}$ is $d$-independent and hence minimally $d$-rigid, and the edge count follows.

In either case, we may apply  \cref{thm:hjns2} to deduce that the graph obtained from $G^-$ by coning with new vertex $v$ and removing $vs_1$ and $vs_2$, and hence $\hat G''$, is $(d+1)$-rigid.
\end{proof}

\subsection{Removing cone edges}

For a vertex $u$ of a graph, write $N(u)$ for its open neighbourhood.

\begin{lemma}\label{lem:remove-one}
  Let $G$ be an $n$-vertex $d$-rigid graph with at least one excess edge,
  and let $u\in V(G)$ satisfy $d(u)\geq n-2$.
  Let $\hat G'$ be the graph obtained by coning with a new vertex $v$
  and then removing the edge $uv$.
  Then $\hat G'$ is $(d+1)$-rigid.
\end{lemma}
\begin{proof}
  Write $x$ for the unique non-neighbour of $u$, if it exists.
  Let $C$ be any $d$-circuit of $G$.
  If $C$ meets $u$ there is nothing to prove so suppose $C$ does not meet $u$.
  If $C$ meets $x$ then consider the set of edges $C'$
  obtained by replacing all edges meeting $x$ by the corresponding edges meeting $u$.
  Since $N(u)\supseteq N(x)$,
  this is well-defined and isomorphic to $C$, so gives a circuit meeting $u$.
  Otherwise, let $y$ be any vertex of $C$
  and observe that $N(u)\cap V(C)\supset N(y)\cap V(C)$.
  Thus we may replace every edge of $C$ meeting $y$
  by the corresponding edge meeting $u$,
  and obtain a circuit $C'$ meeting $u$.
  In either case, the result now follows from \cref{thm:almostcone}.
\end{proof}

Write $N[u]=N(u)\cup\{u\}$ for the closed neighbourhood of a vertex $u$.
Two vertices $u$ and $w$ are \defn{twins}
if $N(u)=N(w)$ or $N[u]=N[w]$.

\begin{lemma}\label{lem:remove-two}
  Let $G$ be an $n$-vertex $d$-rigid graph with at least two excess edges,
  and let $u,w\in V(G)$ be vertices that are not twins but satisfy
  $d(u)=d(w)=n-2$ and $w=\gamma(u)$ for some $\gamma\in \operatorname{Aut}(G)$.
  Let $\hat G''$ be the graph obtained by coning with a new vertex $v$
  and then removing the edges $uv$ and $vw$.
  Then $\hat G''$ is $(d+1)$-rigid.
\end{lemma}
\begin{proof}
  Since $u,w$ are not twins, by \cref{thm:almostcone2},
  it suffices to prove one of the following:
  either every $d$-circuit of $G$ meets both $u$ and $w$,
  or there are $d$-circuits $C_u,C_w$ such that $C_u$ meets $u$,
  $C_w$ meets $w$,
  and either $C_u$ does not meet $w$ or $C_w$ does not meet $u$.
  Let $x$ be the unique non-neighbour of $w$.

  First, suppose that there is a circuit meeting exactly one of $u,w$.
  Applying $\gamma$ or $\gamma^{-1}$ as appropriate gives another circuit meeting the other vertex.
  These two circuits satisfy the second alternative above.

  Secondly, suppose there is a circuit $C$ not meeting $u$ or $w$.
  We claim we can find a distinct circuit $C'$ meeting $w$ but not $u$,
  reducing to the previous case.
  If $x\in V(C)$ then, since $N(x)\subseteq N(w)$,
  we may replace edges meeting $x$ by corresponding edges meeting $w$
  to obtain a suitable $C'$.
  If $x\not\in V(C)$ then choose any $y\in V(C)$;
  since $N(y)\cap V(C)\subset N(w)$,
  we may replace each edge of $C$ meeting $y$ with the corresponding edge meeting $w$.

  If neither of these cases applies,
  then every circuit meets both $u$ and $w$,
  so the first alternative holds.
\end{proof}

\section{Complements of graphs of maximum degree two}\label{sec:two-factors}

In this section we prove \cref{thm:2-reg} by induction on the number of connected components of the complement.
Since the complement has maximum degree at most two,
the base case in which it is connected reduces to paths and cycles.
We treat these two cases first.

Let $G$ be an $n$-vertex graph.
For $d=n-q$, the identity
$dn-\binom{d+1}{2}=\binom n2-\binom q2$
shows that the Maxwell inequality defining $\maxwell(G)$ is equivalent to
$|E(\overline G)|\leq\binom q2$.
By the definition of $a_j$, the least positive integer $q$ satisfying this inequality is
$a_{|E(\overline G)|}+1$. Consequently,
\begin{equation}\label{eq:maxwell-complement}
  \maxwell(G)=n-1-a_{|E(\overline G)|}.
\end{equation}
We record one consequence of~\eqref{eq:maxwell-complement}.
If $H$ is obtained from an $n$-vertex graph $G$ by deleting $s$ vertices,
then $\maxwell(H)\geq\maxwell(G)-s$.
Indeed, $\overline H$ is an induced subgraph of $\overline G$, so
$|E(\overline H)|\leq |E(\overline G)|$ and hence
$a_{|E(\overline H)|}\leq a_{|E(\overline G)|}$.
Therefore~\eqref{eq:maxwell-complement} gives
\begin{equation}\label{eq:maxwell-deletion}
  \maxwell(H)=n-s-1-a_{|E(\overline H)|}
  \geq n-s-1-a_{|E(\overline G)|}=\maxwell(G)-s.
\end{equation}

\subsection{The induction base}

\subsubsection{Cycles}

We first consider the case in which the complement is a cycle.
For $n\geq 3$, let $d_C(n)=\maxwell(\overline C_n)$.
Since $C_n$ has $n$ edges,
\eqref{eq:triangular-index} and~\eqref{eq:maxwell-complement} give the following.

\begin{observation}\label{obs:cycle-dimension}
  For each $n\geq 3$ we have
  $d_C(n)=n-1-a_n=n-\bigl\lceil 1/2+\sqrt{2n+1/4}\bigr\rceil$.
\end{observation}

\begin{lemma}\label{lem:hamilton-step}
  Suppose that $\overline C_n$ is $d$-rigid.
  Then $\overline C_{n+1}$ is $d$-rigid, and it is $(d+1)$-rigid unless $\overline C_n$ is minimally $d$-rigid.
\end{lemma}

\begin{proof}
  Put $G=\overline C_n$.
  If $n=3$, then the hypothesis forces $d=0$ and $G$ is minimally $0$-rigid,
  so the result follows.
  Suppose next that $d=0$.
  We may assume $n\geq 4$, in which case $G$ has an edge and hence is not minimally $0$-rigid.
  The graph $\overline C_{n+1}$ is connected and hence $1$-rigid,
  so the result follows in this case as well.
  We may therefore assume $d\geq 1$, which forces $n\geq 5$.

  First note that
  $|E(G)|=n(n-3)/2\leq n(n-3)-\binom{n-2}{2}$,
  so $d\leq n-3$.
  If $G$ is minimally $d$-rigid,
  choose any edge $xy\in E(C_n)$.
  Otherwise, $G$ has a $d$-circuit.
  If every $d$-circuit of $G$ meets every vertex,
  again choose any edge $xy\in E(C_n)$.
  Otherwise, choose a $d$-circuit $C$ that does not meet every vertex.
  Since $C_n$ is connected,
  it has an edge $xy$ such that $C$ meets $x$ but not $y$.

  We construct $\overline C_{n+1}$ from $\overline C_n$ in two steps:
  first add the edge $xy$, then add a new vertex adjacent to everything except $x$ and $y$.
  Note that the second step adds a vertex of degree $n-2\geq d+1$.
  Thus, by \cref{lem:zero-extension},
  $\overline C_{n+1}$ is always $d$-rigid,
  and if the graph created by adding an edge was $(d+1)$-rigid
  then so is $\overline C_{n+1}$.
  It remains to prove the latter conclusion when $G$ is not minimally $d$-rigid
  and $G+xy$ is not already $(d+1)$-rigid.

  Since $G$ is $d$-rigid,
  the graph $G+xy$ contains a $d$-circuit $D$ containing $xy$.
  Moreover, $G+xy$ has at least two excess edges.
  If every $d$-circuit of $G$ meets every vertex,
  then every $d$-circuit of $G+xy$ meets both $x$ and $y$:
  a circuit not containing $xy$ is a circuit of $G$,
  while a circuit containing $xy$ meets both vertices.
  Otherwise, the circuit $C$ chosen above meets $x$ but not $y$,
  whereas $D$ meets $y$.

  Let $x'$ and $y'$ be the other neighbours of $x$ and $y$ in $C_n$, respectively.
  Since $n\geq 5$, we have $x'\neq y'$,
  and hence
  $N_{G+xy}(x)-y\neq N_{G+xy}(y)-x$.
  Thus the conditions of \cref{thm:almostcone2} are satisfied.
  Coning $G+xy$ with a new vertex $v$ and deleting $vx$ and $vy$
  gives $\overline C_{n+1}$,
  so this graph is $(d+1)$-rigid.
\end{proof}

\begin{theorem}\label{thm:hamilton}
  For each $n\geq3$, the graph $\overline C_n$ is $d_C(n)$-rigid.
\end{theorem}

\begin{proof}
  Recall that
  $d_C(n)=n-\bigl\lceil 1/2+\sqrt{2n+1/4}\bigr\rceil$.
  We use induction on $n$; the result holds trivially for $n\leq 4$.
  Write $n=\binom{r+1}{2}-s$ for some $0\leq s< r$; since $\binom r2=\binom{r+1}{2}-r$,
  a unique such $r,s$ exists for any $n$.
  We have $r$ minimal such that $r(r+1)\geq 2n$,
  i.e.\ $r=\bigl\lceil-1/2+\sqrt{2n+1/4}\bigr\rceil$.
  Set $d=d_C(n)=n-r-1$ and assume that $\overline C_n$ is $d$-rigid.
  Then
    $dn-\binom {d+1}{2}
     =n(n-r-1)-\binom{n-r}{2}
     =\frac{n^2-n}{2}-\binom{r+1}{2}
     =\frac{n^2-3n}{2}-s$.
  Thus, $\overline C_n$ has $s$ excess edges.

  Letting $r',s'$ be such that $n+1=\binom{r'+1}{2}-s'$,
  we have $r'=r$ and $s'=s-1$ unless $s=0$, when $r'=r+1$ and $s'=r$.
  By \cref{lem:hamilton-step}, if $s>0$ then $\overline C_{n+1}$ is $(d+1)$-rigid,
  and $d+1=d_C(n+1)=n+1-r'-1$, as required.
  Otherwise, $\overline C_{n+1}$ is $d$-rigid,
  and $d=d_C(n+1)=n+1-r'-1$, as required.
\end{proof}

\subsubsection{Paths}

We next consider the case in which the complement is a path.
For $n\geq 1$, let $d_P(n)=\maxwell(\overline P_n)$.
Since $P_n$ has $n-1$ edges,
\eqref{eq:triangular-index} and~\eqref{eq:maxwell-complement} give the following.

\begin{observation}\label{obs:path-dimension}
  For each $n\geq 1$ we have
  $d_P(n)=n-1-a_{n-1}=n-\bigl\lceil 1/2+\sqrt{2n-7/4}\bigr\rceil$.
\end{observation}

\begin{lemma}\label{lem:path-step}
  Suppose that $\overline P_n$ is $d$-rigid.
  Then $\overline P_{n+1}$ is $d$-rigid,
  and it is $(d+1)$-rigid unless $\overline P_n$ is minimally $d$-rigid.
\end{lemma}

\begin{proof}
  Put $G=\overline P_n$.
  If $d=0$, then $G$ is minimally $0$-rigid precisely when $n\leq 2$.
  For $n\geq 3$, the graph $\overline P_{n+1}$ is connected and hence $1$-rigid,
  so the result follows.
  We may therefore assume $d\geq 1$.

  First note that $d\leq n-2$.
  The graph $\overline P_{n+1}$ is obtained from $G$
  by coning with a new vertex $v_{n+1}$
  and removing the edge $v_nv_{n+1}$,
  where $v_n$ is an endpoint of $P_n$ and
  $d_G(v_n)=n-2$.
  The new vertex has degree $n-1\geq d+1$,
  so \cref{lem:zero-extension} implies that $\overline P_{n+1}$ is $d$-rigid.
  If $G$ is not minimally $d$-rigid,
  then it has at least one excess edge,
  and \cref{lem:remove-one} implies that $\overline P_{n+1}$ is $(d+1)$-rigid.
\end{proof}

\begin{theorem}\label{thm:path}
  For each $n\geq1$, the graph $\overline P_n$ is $d_P(n)$-rigid.
\end{theorem}

\begin{proof}
  We use induction on $n$, with the case $n=1$ being trivial.
  Set $d=d_P(n)$ and assume that $\overline P_n$ is $d$-rigid.
  We have
  \[
    |E(\overline P_{n+1})|=|E(\overline P_n)|+n-1
  \]
  and
  \[
    (d+1)(n+1)-\binom{d+2}{2}
    =dn-\binom{d+1}{2}+n.
  \]
  Since $\overline P_n$ is obtained from $\overline P_{n+1}$ by deleting one vertex,
  \eqref{eq:maxwell-deletion} gives $d_P(n+1)\leq d+1$.
  By the definition of $d=d_P(n)$, we have
  $|E(\overline P_n)|\ge dn-\binom{d+1}{2}$.
  Since $|E(\overline P_{n+1})|=|E(\overline P_n)|+n-1$
  and $n-1\ge d$, it follows that $d_P(n+1)\ge d$.
  Comparing the two displayed identities then shows that the upper value $d+1$
  occurs precisely when $\overline P_n$ has an excess edge. Thus,
  \[
    d_P(n+1)=
    \begin{cases}
      d,&\text{if }|E(\overline P_n)|=dn-\binom{d+1}{2},\\
      d+1,&\text{otherwise.}
    \end{cases}
  \]
  In the first case $\overline P_n$ is minimally $d$-rigid,
  and in the second case it is not.
  The result now follows from \cref{lem:path-step}.
\end{proof}

\subsection{The induction step}

The \defn{join} $G\vee H$ of two graphs $G$ and $H$ is obtained from their disjoint union by adding all edges between $V(G)$ and $V(H)$.

We next show how to add a path or a cycle to the complement.
\begin{lemma}\label{lem:add-cycle}
Let $G_0$ be a non-complete $n$-vertex $d_0$-rigid graph, and let $F$ be either $P_k$ or $C_k$, where in the latter case we require $k\geq 5$. Then $G_0\vee \overline F$ is $d$-rigid for any $d$ satisfying both of the following:
		\begin{itemize}
			\item $d\leq d_0+k$; and
			\item $G_0\vee \overline F$ has at least $d(n+k)-\binom{d+1}{2}$ edges.
		\end{itemize}
\end{lemma}

\begin{proof}
  We may assume $d_0=\rig(G_0)$,
  since $d$ will still satisfy the conditions if $d_0$ is increased.
  Write $V(F)=\{v_1,\ldots,v_k\}$ where $v_1v_2,\ldots,v_{k-1}v_k\in E(F)$.
  We construct $G_0\vee \overline F$ from $G_0$
  by adding vertices of $\overline F$ one by one in order of index.
  Let $G_i$ be the graph immediately after adding $v_i$,
  so that $|V(G_i)|=n+i$ and $G_k=G_0\vee \overline F$.
  When adding a vertex $v_i$, we first cone with new vertex $v_i$, then remove any edge(s) to its neighbour(s) in $F$ previously added.
  Note that
  \begin{itemize}
    \item adding vertex $v_1$ does not require removing any edges;
    \item adding vertex $v_i$ for $1<i<k$ requires removing the edge $v_iv_{i-1}$, and $d_{G_{i-1}}(v_{i-1})\geq|V(G_{i-1})|-2$;
    \item if $F\cong P_k$ ($k\ge 2$) then adding vertex $v_k$ requires removing the edge $v_kv_{k-1}$, and $d_{G_{k-1}}(v_{k-1})\geq|V(G_{k-1})|-2$;
    \item if $F\cong C_k$ then adding vertex $v_k$ requires removing the edges $v_kv_{k-1}$ and $v_kv_1$,
      we have $d_{G_{k-1}}(v_1)=d_{G_{k-1}}(v_{k-1})=|V(G_{k-1})|-2$,
      the permutation reversing the order of $v_1,\ldots,v_{k-1}$
      and leaving other vertices fixed is an automorphism of $G_{k-1}$,
      and, since $k\geq 5$, $v_1$ and $v_{k-1}$ are not twins of $G_{k-1}$.
  \end{itemize}
  For $i\ge 1$,
  let $d_i=\rig(G_i)$.
  Note that, since $G_i$ is a subgraph of the cone of $G_{i-1}$,
   \cref{thm:coning} implies that $d_i\leq d_{i-1}+1$.
  We claim that, for each $1\leq i\leq k$, we have $d_i=d_{i-1}+1$ if
  \begin{equation}|E(G_i)|\geq(d_{i-1}+1)(n+i)-\binom{d_{i-1}+1+1}{2},\label{eq:inequality}\end{equation}
  and $d_i=d_{i-1}$ otherwise.
  Indeed, if \eqref{eq:inequality} is satisfied then $G_{i-1}$ had at least $t$ excess edges,
  where $t$ is the number of edges removed when adding $v_i$.
  If $t=0$ the claim follows from \cref{thm:coning}.
  Otherwise, in each case listed above the conditions of \cref{lem:remove-one}
  or \cref{lem:remove-two} are satisfied,
  and the claim follows.
  If \eqref{eq:inequality} is not satisfied,
  observe that $G_0$ being non-complete implies $d_0\leq n-2$,
  and hence $d_{i-1}\leq n+i-3$.
  Since $G_i$ is obtained by adding a vertex of degree at least $n+i-3$
  to the $d_{i-1}$-rigid graph $G_{i-1}$,
  \cref{lem:zero-extension} implies that $G_i$ is also $d_{i-1}$-rigid
  and the claim holds.

  Finally, we argue that it follows from the claim above that $d_k\geq d$.
  Indeed, suppose that $d_k= d_0+k-\ell$ for some $\ell>0$.
  Then there is some smallest $i$ such that $d_i\leq d_0+i-\ell$,
  and by the claim we must have 
  \begin{equation}|E(G_i)|<(d_0+i-\ell+1)(n+i)-\binom{d_0+i-\ell+2}{2}.\label{H_i}\end{equation}
  Suppose \eqref{H_i} is true for some $i<k$.
  Now we have \[|E(G_{i+1})|\leq |E(G_i)|+n+i<(d_0+(i+1)-\ell+1)(n+i+1)-\binom{d_0+(i+1)-\ell+2}{2}.\]
  Thus \eqref{H_i} also holds with $i$ replaced by $i+1$, and by induction it holds for $i=k$, giving $d<d_k+1$.
\end{proof}
	We now conclude the main result of this section.
\begin{proof}[Proof of \cref{thm:2-reg}]
  Let $d=\maxwell(G)$.
  The result is trivial if $G$ is complete, so we assume not.
  We prove the conclusion for all $n$ by induction on $r$,
  where $\overline G$ has $r$ components.
  The case $r=1$ is covered by \cref{thm:hamilton} and \cref{thm:path}.

  We may assume $r\geq 2$.
  Choose a component $F$ of $\overline G$ such that
  $G'=G-V(F)$ is non-complete.
  This is possible because $G$ is non-complete:
  if $\overline G$ has only one nontrivial component, choose an isolated vertex for $F$,
  and otherwise choose $F$ so that another nontrivial component remains.
  Let $k=|V(F)|$.
  Then $F$ is either $P_k$ or $C_k$,
  where in the latter case $k\geq 5$ by the hypothesis of \cref{thm:2-reg}.
  We have $G=G'\vee \overline F$.

  Let $d'=\maxwell(G')$.
  By induction, $G'$ is $d'$-rigid.
  Since $G'$ is obtained from $G$ by deleting $k$ vertices,
  \eqref{eq:maxwell-deletion} gives $d\leq d'+k$.
  Thus we conclude from \cref{lem:add-cycle} that
  $G=G'\vee \overline F$ is $d$-rigid.
\end{proof}

\section{Complements of bounded-degree graphs}\label{sec:bounded-degree}

We begin this section by proving \cref{cor:exceptional-components}.

\begin{proof}[Proof of \cref{cor:exceptional-components}]
  Put $F=\overline G$, $t=\sigma(F)$ and $d=\maxwell(G)$.
  Since $\Delta(F)\leq 2$,
  the value $t$ is precisely the number of $C_3$- and $C_4$-components of $F$.
  If $t=0$, the result follows from \cref{thm:2-reg}.

  Choose a set $S$ containing one vertex from each $C_3$- or $C_4$-component of $F$,
  and put $H=G-S$ and $n'=|V(H)|=n-t$.
  Each of the $t$ exceptional components has at least $3$ vertices, so $n\geq 3t$.
  If $t\geq 2$, then $n'=n-t\geq 2t\geq 4$.
  If $t=1$ and $n=4$, then $|E(G)|\leq 3$,
  so $d\leq 1$ and the result follows from the $0$-rigidity convention.
  We may therefore assume that $n'\geq 4$.

  The graph $F-S$ has maximum degree at most two and no component isomorphic to $C_3$ or $C_4$.
  Thus $H$ is Maxwell-sharp by \cref{thm:2-reg}.
  Put $d'=\max\{d-t,0\}$.
  By~\eqref{eq:maxwell-deletion}, $\maxwell(H)\geq d-t$.
  Since $H$ is Maxwell-sharp,
  downward monotonicity and the $0$-rigidity convention imply that $H$ is $d'$-rigid.

  If $t=1$, then $|E(F)|\geq 3$,
  so \eqref{eq:maxwell-complement} gives $d\leq n-3$ and hence $d'\leq n'-3$.
  The unique vertex of $S$ has $n'-2\geq d'$ neighbours in $H$,
  so \cref{lem:zero-extension} implies that $G$ is $d'$-rigid.
  Since $d'\geq d-1=d-\lceil t/2\rceil$, the result follows in this case.
  We may therefore assume that $t\geq 2$.

  Pair all but at most one vertex of $S$,
  and let $M$ be the resulting matching of size $\floor{t/2}$.
  For each edge of $M$, no path of length at most $2$ joins its ends in $F$.
  Since $|E(F)|\geq 6$,
  \eqref{eq:maxwell-complement} gives $d\leq n-4$ and hence $d'\leq n'-4$.
  Hence \cref{lem:paired-extension}, applied in dimension $d'$ with $\Delta(F)=2$,
  shows that $G$ is rigid in dimension at least
  $d'+\floor{t/2}\geq d-t+\floor{t/2}=d-\bigl\lceil t/2\bigr\rceil$.
\end{proof}

Recall that, for a graph $F$, a set $S\subseteq V(F)$ is admissible if $F-S$ has maximum degree at most two
and no connected component isomorphic to $C_3$ or $C_4$.

\begin{theorem}\label{thm:structural-bound}
  Let $D=D(n)$ be an integer with $D\ge 2$,
  and let $F$ be an $n$-vertex graph with $\Delta(F)\leq D$.
  Let $S\subseteq V(F)$ be admissible.
  Put $s=|S|$, $n'=n-s$ and $H=\overline F-S$.
  If $n'\ge 2D$, then
  \[
    \rig(\overline{F})
    \ge \min\{\maxwell(H),n'-2D\}
    +\frac{1}{2}\max\{0,s-1-D^2\}.
  \]
\end{theorem}

\begin{proof}
Put $d=\min\{\maxwell(H),n'-2D\}\ge 0$
and $m=\max\{0,s-1-D^2\}$.
We first find a $d$-rigid induced subgraph on $n'$ vertices
and then add at least $m$ of the remaining vertices in far-apart pairs.
By admissibility, $\overline H$ has maximum degree at most two and no component isomorphic to $C_3$ or $C_4$.
Since $n'\ge 2D\ge 4$,
$H$ is Maxwell-sharp by \cref{thm:2-reg}.
Since $d\leq\maxwell(H)$,
downward monotonicity implies that $H$ is $d$-rigid.

Define an auxiliary graph $X$ on $S$ by joining two vertices if no path of length at most $2$ joins them in $F$.
For every $x\in S$,
at most $D+D(D-1)=D^2$ vertices of $S\sm\{x\}$ can be joined to $x$ by a path of length at most $2$ in $F$.
Thus every vertex of $X$ has degree at least $s-1-D^2$.
Let $M$ be a maximal matching in $X$.
Since the vertices not covered by $M$ form a clique in $\overline X$ and $\Delta(\overline X)\leq D^2$,
at most $D^2+1$ vertices are uncovered.
Therefore $2|M|\ge\max\{0,s-1-D^2\}=m$.
Since $d\leq n'-2D\leq n'-2\Delta(F)$,
\cref{lem:paired-extension} shows that $\overline F$ is $(d+|M|)$-rigid.
Hence $\rig(\overline F)\ge d+|M|\ge d+m/2$, as required.
\end{proof}

Recall that a linear forest is a graph whose components are paths, possibly trivial.
Akbari, Amanihamedani, Mousavi, Nikpey and Sheybani~\cite{AAMNS19+} conjectured that every graph $G$ with $\delta(G)\geq2$ has an induced linear forest with at least
\mbox{$\sum_{v\in V(G)}2/(d_G(v)+1)$} vertices.
Joret and Petit~\cite{JP25} proved this as a consequence of the following stronger theorem.
Given a function $f\colon\mathbb{N}\to[0,1]$ and a graph $G$,
with a slight abuse of notation write
\[
  f(G)=\sum_{v\in V(G)}f(d_G(v)).
\]

\begin{theorem}[Joret and Petit~\cite{JP25}*{Theorem~4}]\label{thm:joret-petit}
  Let $f\colon\mathbb{N}\to\mathbb{R}$ be defined as follows:
  \[
    f(d)=
    \begin{cases}
      1&\text{if $d=0$},\\
      5/6&\text{if $d=1$},\\
      2/(d+1)&\text{if $d\geq2$}.
    \end{cases}
  \]
  Then every graph $G$ has an induced linear forest with at least $f(G)$ vertices.
\end{theorem}

Applying \cref{thm:structural-bound} to the induced linear forest supplied by \cref{thm:joret-petit}
gives the following exact degree-sequence bound.

\begin{corollary}\label{cor:degree-sequence-bound-exact}
  Let $D=D(n)$ be an integer with $D\ge 2$,
  and let $F$ be an $n$-vertex graph with $\Delta(F)\leq D$.
  Put $n'=\ceil{f(F)}$ and $s=n-n'$, and suppose that $n'\ge 2D$.
  Then
  \[
    \rig(\overline F)
    \ge \min\{d_P(n'),n'-2D\}
    +\frac{1}{2}\max\{0,s-1-D^2\}.
  \]
\end{corollary}

\begin{proof}
By \cref{thm:joret-petit},
there is a set $U\subseteq V(F)$ of size $n'$ such that $F[U]$ is a linear forest.
Put $S=V(F)\sm U$ and $H=\overline F-S$.
Then $|S|=s$, and $S$ is admissible.
Moreover,
$|E(\overline H)|\leq n'-1=|E(P_{n'})|$.
Thus $|E(H)|\geq |E(\overline P_{n'})|$,
and monotonicity of $\maxwell$ in the number of edges gives
$\maxwell(H)\ge\maxwell(\overline P_{n'})=d_P(n')$.
The result now follows from \cref{thm:structural-bound}.
\end{proof}

We will prove the following exact form of \cref{cor:bounded-complement-degree},
which immediately implies that result.
\begin{corollary}
  \label{cor:bounded-complement-degree-exact}
  Let $\Delta\ge 2$ be an integer,
  and let $G$ be an $n$-vertex graph with $\delta(G)\ge n-1-\Delta$.
  If $n\ge 2\Delta^2$, then
  \[
    \rig(G)
    \ge d_P(n)-\frac{n}{2}+\frac{n}{\Delta+1}
    -\frac{\Delta^2+1}{2}.
  \]
\end{corollary}

\begin{proof}
  Let $F$ be the complement of $G$.
  The assumption $\delta(G)\ge n-1-\Delta$ gives $\Delta(F)\le\Delta$.
  Put $D=\max\{2,\Delta(F)\}$ and $n'=\ceil{f(F)}$.
  Since $f(k)\ge 2/(\Delta+1)$ for every $0\le k\le\Delta$,
  the assumption $n\ge 2\Delta^2$ gives
  $n'\ge 2n/(\Delta+1)\ge 2\Delta\ge 2D$.
  Moreover, since $n\ge 2\Delta^2$, \cref{obs:path-dimension} gives
  \begin{equation}\label{eq:path-deficit-bound}
    n-d_P(n)
    =\ceil{\frac{1}{2}+\sqrt{2n-\frac{7}{4}}}
    \ge \frac{1}{2}+\sqrt{4\Delta^2-\frac{7}{4}}
    \ge 2\Delta
    \ge 2D.
  \end{equation}
  The penultimate inequality follows because
  $4\Delta^2-\frac{7}{4}-(2\Delta-\frac{1}{2})^2=2\Delta-2\ge 0$,
  and the last follows from $D\le\Delta$.
  Since $\overline P_{n'}$ is obtained from $\overline P_n$ by deleting $n-n'$ vertices,
  \eqref{eq:maxwell-deletion} gives
  $d_P(n')\ge d_P(n)-(n-n')=n'-(n-d_P(n))$.
  By~\eqref{eq:path-deficit-bound},
  $n'-2D\ge n'-(n-d_P(n))$.
  Thus,
  \[
    \min\{d_P(n'),n'-2D\}\ge n'-(n-d_P(n)).
  \]
  Also, since $D\le\Delta$,
  $\max\{0,n-n'-1-D^2\}\ge n-n'-1-\Delta^2$.
  Therefore \cref{cor:degree-sequence-bound-exact} gives
  \begin{align*}
    \rig(G)
    &\ge n'-(n-d_P(n))+\frac{n-n'-1-\Delta^2}{2}\\
    &=d_P(n)-\frac{n}{2}+\frac{n'}{2}-\frac{\Delta^2+1}{2}
    \ge d_P(n)-\frac{n}{2}+\frac{n}{\Delta+1}-\frac{\Delta^2+1}{2},
  \end{align*}
  as required.
\end{proof}

\begin{proof}[Proof of \cref{cor:bounded-complement-degree}]
  Since $\Delta=o(\sqrt n)$,
  we have $n\ge 2\Delta^2$ for all sufficiently large $n$.
  By \cref{obs:path-dimension}, $d_P(n)=n-O(\sqrt n)$,
  so \cref{cor:bounded-complement-degree-exact} gives the stated bound.
\end{proof}

\section*{Acknowledgements}

A.\,N.\ was partially supported by EPSRC grant EP/X036723/1 and by UK Research and Innovation (grant number UKRI1112), under the EPSRC Mathematical Sciences Small Grant scheme.
P.\,M.\ was supported by UKRI under the Horizon Europe Guarantee for MSCA Postdoctoral Fellowships (EP/Z001781/1, IRIS).
For the purpose of open access, the authors have applied a Creative Commons Attribution (CC-BY) licence to any Author Accepted Manuscript version arising from this work.

	\bibliography{library}

\end{document}